\documentclass[11pt]{article}
\usepackage{mathptmx}
\usepackage[latin9]{inputenc}
\usepackage{color}
\usepackage{mathrsfs}
\usepackage{mathtools}
\usepackage{amsmath}
\usepackage{amsthm}
\usepackage{amssymb}
\usepackage{graphicx}
\usepackage[unicode=true,pdfusetitle,
 bookmarks=true,bookmarksnumbered=false,bookmarksopen=false,
 breaklinks=false,pdfborder={0 0 0},pdfborderstyle={},backref=false,colorlinks=true]
 {hyperref}
\hypersetup{
 pdfborderstyle={}}
\usepackage{breakurl}

\makeatletter
\theoremstyle{plain}
\newtheorem{thm}{\protect\theoremname}
  \theoremstyle{plain}
  \newtheorem{lem}[thm]{\protect\lemmaname}
  \theoremstyle{plain}
  \newtheorem{cor}[thm]{\protect\corollaryname}

\@ifundefined{date}{}{\date{}}
\usepackage{amsfonts}
\usepackage{dsfont}

\numberwithin{equation}{section}

\allowdisplaybreaks[4]
\DeclareMathAlphabet{\mathcal}{OMS}{cmsy}{m}{n}

  \providecommand{\corollaryname}{Corollary}
  \providecommand{\lemmaname}{Lemma}
  
\providecommand{\theoremname}{Theorem}

\makeatother

  \providecommand{\corollaryname}{Corollary}
  \providecommand{\lemmaname}{Lemma}
\providecommand{\theoremname}{Theorem}

\begin{document}

\title{Optimal probabilities and controls for reflecting diffusion processes}

\author{Zhongmin Qian\thanks{Mathematical Institute, University of Oxford, United Kingdom. Email:
qianz@maths.ox.ac.uk} \ and\ Xingcheng Xu\thanks{School of Mathematical Sciences, Peking University, Beijing, China;
Current address: Mathematical Institute, University of Oxford, United
Kingdom. Email: xuxingcheng@pku.edu.cn} \thanks{Xingcheng Xu is supported by China Scholarship Council, Grant No.
201706010019.}}
\maketitle
\begin{abstract}
A solution to the optimal problem for determining vector fields which
maximize (resp. minimize) the transition probabilities from one location
to another for a class of reflecting diffusion processes is obtained
in the present paper. The approach is based on a representation for
the transition probability density functions. The optimal transition
probabilities under the constraint that the drift vector field is
bounded by a constant are studied in terms of the HJB equation. In
dimension one, the optimal reflecting diffusion processes and the
bang-bang diffusion processes are considered. We demonstrate by simulations
that, even in this special case, by considering the nodal set of the
solutions to the HJB equation, the optimal diffusion processes exhibit
an interesting feature of phase transitions. An optimal stochastic
control problem for a class of stochastic control problems involving
diffusion processes with reflection is also solved in the same spirit. 
\end{abstract}
\vspace{5mm}
 \hspace{11mm}\textbf{Keywords:} Reflecting diffusion, Comparison
theorem, Optimal transition probability density,

\hspace{5mm}Cameron-Martin formula, Stochastic optimal control. \vspace{5mm}

\hspace{5mm}\textbf{MSC(2010):} Primary: 60H10, 60H30; Secondary:
49J30, 93E20.

\section{Introduction}

The simple optimal control problem to determine vector fields $b(t,x)$
bounded by a constant $\kappa\geq0$ which maximize (resp. minimize)
the probability $p_{b}(s,x;t,y)$ of diffusion processes
\begin{equation}
dX_{t}=b(t,X_{t})dt+dB_{t}\label{eq:d1}
\end{equation}
started at $X_{s}=x$ and ended at $X_{t}=y$ (where $B=(B_{t})_{t\geq0}$
is a Brownian motion ) has been considered and solved explicitly in
the previous work \cite{KS84,KS85,QZ02,QRZ03,QZ04}. The method utilized
in \cite{QZ02,QRZ03} is quite elementary and is based on the density
version of the Cameron-Martin formula
\begin{equation}
p_{b+c}(s,x;t,y)=p_{b}(s,x;t,y)+\int_{s}^{t}\mathbb{E}_{s,x}\left\{ R_{s,r}c(r,X_{r})\cdot\nabla_{x}p_{b}(r,X_{r};t,y)\right\} dr\label{eq:rep1}
\end{equation}
for $0\leq s<t$, where $p_{b}(s,x;t,y)$ denotes the transition probability
density of $X_{t}$ defined by (\ref{eq:d1}) under the condition
that $X_{s}=x$ with respect to the Lebesgue measure. Here $b(t,x)$
and $c(t,x)$ are two vector fields with at most linear growth, $(X_{t},\mathbb{P}_{s,x})$
is the weak solution to (\ref{eq:d1}) in the sense of Stroock-Varadhan's
article \cite{SW1971-boundary}, and $R_{s,r}$ is the Cameron-Martin
density process
\begin{equation}
R_{s,t}=\exp\left[\int_{s}^{t}c(r,X_{r})dW_{r}-\frac{1}{2}\int_{0}^{t}|c|^{2}(r,X_{r})dr\right],\label{eq:3.5-1}
\end{equation}
where $W$ is the martingale part of $X$. A simple inspection gives
the optimal solutions $b(t,x)=\pm\kappa(x-y)/|x-y|$, to which an
explicit formula, in dimension one, for $p_{b}(s,x;t,y)$ is given
in \cite{KS84,QZ02}. 

The question becomes difficult if we consider the simple optimal control
problem for diffusion processes with barriers, which arise from many
stochastic optimization problems for example in pricing problems for
options. 

Let $G\subseteq\mathbb{R}^{n}$ be a domain with a smooth boundary
$\partial G$, and $\bar{G}$ denote its closure. We wish to locate
a vector field $b(t,x)$ (for $t\geq0$ and $x\in\bar{G}$) bounded
by $\kappa$, which maximizes (resp. minimizes) the probability $q_{b}(s,x;t,y)$
(where $0\leq s<t,$ $x,y\in\bar{G}$) of reflecting diffusion processes
\begin{equation}
dX_{t}=b(t,X_{t})dt+dB_{t}+dL_{t}\label{eq:ref-d2}
\end{equation}
started at $X_{s}=x\in\bar{G}$ and finished at $X_{t}=y\in\bar{G}$,
where $B=(B_{t})$ is a Brownian motion in $\mathbb{R}^{n}$, $L$
is the local time of $X$ with respect to the boundary $\partial G$,
so that $t\rightarrow L_{t}$ increases only on $\left\{ t:X_{t}\in\partial G\right\} $.
In this paper, we are going to establish the following
\begin{thm}
\label{thm:Theorem1}Let $\kappa\geq0$ be a constant. Given $y\in\bar{G}$
and $T>0$. Let $u^{\pm}(t,x)$ (where $t\geq0$ and $x\in\bar{G}$)
be the unique solution to the terminal and boundary problem of the
backward parabolic equation
\begin{equation}
\begin{cases}
\frac{\partial}{\partial t}u+\frac{1}{2}\Delta u\pm\kappa|\nabla u|=0, & \textrm{for}\ 0\leq t<T,\ x\in G\\
\lim_{t\uparrow T}u(t,x)=\delta_{y}(x), & \textrm{for}\ x\in\bar{G}\\
\left.\frac{\partial}{\partial\nu}u(t,\cdot)\right|_{\partial G}=0, & \textrm{for}\ 0\leq t\leq T.
\end{cases}\label{eq:semilinear-PDE-1}
\end{equation}
Define 
\[
b_{\kappa}^{\pm}(t,x)=\pm\kappa\frac{\nabla u^{\pm}(t\wedge T,x)}{\left|\nabla u^{\pm}(t\wedge T,x)\right|}
\]
for $t\geq0$ and $x\in\bar{G}$. Let $q_{b}(s,x;t,y)$ be the transition
probability density of the diffusion defined by (\ref{eq:ref-d2}),
where $b(t,x),$ defined on $[0,\infty)\times\bar{G}$, is a bounded,
Borel measurable vector field such that $|b(t,x)|\leq\kappa$ for
$t\geq0$ and $x\in\bar{G}$. Then
\begin{equation}
q_{b_{\kappa}^{-}}(t,x;T,y)\leq q_{b}(t,x;T,y)\leq q_{b_{\kappa}^{+}}(t,x;T,y)\label{eq:com1}
\end{equation}
for all $0\leq t\leq T$ and $x\in\bar{G}$.
\end{thm}

Obviously, for given $T$ and $y$, the bounds in (\ref{eq:com1})
for $q_{b}(t,x;T,y)$ is optimal, and (\ref{eq:semilinear-PDE-1})
can be considered as the Hamilton-Jacobi-Bellman (HJB) equation for
the optimization problem for $q_{b}(t,x;T,y)$.

The semi-linear parabolic equations such as (\ref{eq:semilinear-PDE-1})
have been studied in PDE literature (see e.g. \cite{LadyzhenskajaParabolic1968}).
In order to carry out explicit computations, one needs to consider
the nodal set of the space-derivative $\nabla u(t,x)$, which also
solves a non-linear parabolic equation. The study of nodal sets of
solutions to semi-linear parabolic equations is however a difficult
subject, and is far from complete. Interesting results may be found
in the papers \cite{lin-nodal1991,lin-nodal1994} and etc.

In the case that $G=\mathbb{R}^{n}$, given $T>0$ and $y\in\mathbb{R}^{n}$
then $b^{\pm}(t,x)=\mp\kappa(x-y)/|x-y|$, the radial direction vector
fields, which have been determined in \cite{QRZ03,QZ04}. Here we
propose a new method for determining the HJB equations for this optimization
problem based on a representation for the perturbations of reflecting
diffusion processes, which extends the approach in \cite{QRZ03} to
reflecting diffusion processes.

There is of course huge literature both on diffusion processes and
related stochastic optimal control problems, for the general aspects
of their study, the reader should refer to the standard references
such as \cite{Davies1989Heatkernel,Ikeda-Watanabe,Ito-McKean1974,Krylov09,KS91,Mansuy-Yoy2008,RY99,SV79}. 

The paper is organized as following. In the section \S\ref{sec:Optimal-bounds},
we establish a representation formula for the transition probability
density of the reflecting diffusion process. Then, we present the
proof of Theorem \ref{thm:Theorem1} by the study of the representation
and the HJB equation. In the section \S\ref{sec:Reflecting-bang-bang-diffusion},
we consider the one dimensional case with $G=[0,\infty)$, and we
give the explicit formula of the optimal transition probability densities
for the case $y=0$. We also study the connection with the reflecting
bang-bang diffusion process. In order to gain further knowledge about
the optimal transition probabilities $q_{b_{\kappa}^{\pm}}(t,x;T,y)$
for the general case, for example, $y>0$ and $G=[0,\infty)$, we
demonstrate, in the section \S\ref{sec:HJB}, by numerical simulations
that the optimal diffusion processes exhibit an interesting feature
of phase transitions. Hence, the HJB equation may be equivalent to
a free boundary problem. We study a solvable stochastic control problem
for a class of diffusion type processes with reflection in the section
\S\ref{sec:Stochastic-Control}. We find out the optimal process
and calculate its transition probability, which is connected with
the optimal process in the section \S\ref{sec:Reflecting-bang-bang-diffusion}.
The explicit formula of the value functions are also given there.

\section{Optimal bounds for reflecting diffusion processes\label{sec:Optimal-bounds}}

This section is devoted to the proof of Theorem \ref{thm:Theorem1}. 

The main ingredient in the proof of Theorem \ref{thm:Theorem1} is
a density version of the Cameron-Martin formula for reflecting diffusion
processes. Let $G\subseteq\mathbb{R}^{n}$ be an open subset with
a smooth boundary $\partial G$, and $\nu$ denote the outer unit
normal vector fields along $\partial G$. Suppose $b(t,x)$ and $c(t,x)$
are two bounded (time-dependent) vector fields for $t\geq0$ and $x\in\bar{G}$.
Let $\left(X_{t},\mathbb{P}_{s,x}\right)$ be the reflecting diffusion
process with infinitesimal generator 
\[
\mathscr{L}_{t,x}=\frac{1}{2}\Delta+b(t,x)\cdot\nabla
\]
with its state space $\bar{G}$, that is, $\mathbb{P}_{s,x}$ (for
every $s\geq0$ and $x\in\bar{G}$) is the solution to the martingale
problem (see e.g. \cite{SW1971-boundary}):
\[
M_{t}^{[f]}=f(t,X_{t})-f(s,X_{s})-\int_{s}^{t}\mathscr{L}_{r,X_{r}}f(r,X_{r})dr
\]
is a local martingale (where $t\geq s$) for every $f\in C_{b}^{1,2}([0,\infty)\times\bar{G})$
such that $\left.\frac{\partial}{\partial\nu}f(t,\cdot)\right|_{\partial G}=0$
as for all $t>0$. Define a family of probability measures $\mathbb{Q}_{s,x}$
by 
\begin{equation}
\frac{d\mathbb{Q}_{s,x}}{d\mathbb{P}_{s,x}}\bigg|_{\mathcal{F}_{t}}=R_{s,t}:=\exp\left\{ \int_{s}^{t}c(r,X_{r})\cdot dW_{r}-\frac{1}{2}\int_{s}^{t}|c|^{2}(r,X_{r})dr\right\} ,\label{eq:careron-martin-Q}
\end{equation}
where $s\leq t$, and $W$ is the martingale part of $X$ which is
a Brownian motion in $\mathbb{R}^{n}$ under $\mathbb{P}_{s,x}$. 
\begin{lem}
\label{lem:Lemma2}Under above assumptions and notations. $(X_{t},\mathbb{Q}_{s,x})$
(for $s\geq0$ and $x\in\bar{G}$) is a reflecting diffusion process
with its infinitesimal generator 
\[
\tilde{\mathscr{L}}_{t,x}=\frac{1}{2}\Delta+(b(t,x)+c(t,x))\cdot\nabla.
\]
That is, for any pair $s\geq0$ and $x\in\bar{G}$,
\[
\tilde{M}_{t}^{[f]}=f(t,X_{t})-f(s,X_{s})-\int_{s}^{t}\tilde{\mathscr{L}}_{r,X_{r}}f(r,X_{r})dr
\]
is a local martingale for $t\geq s$ under the probability $\mathbb{Q}_{s,x}$,
for every $f\in C_{b}^{1,2}\left([0,\infty)\times\bar{G}\right)$
such that $\left.\frac{\partial}{\partial\nu}f(t,\cdot)\right|_{\partial G}=0$
for all $t>0$. 
\end{lem}

\begin{proof}
Without losing generality, we may assume that $s=0$ and $x\in\bar{G}$
is fixed. Under $\mathbb{P}_{0,x}$, $M^{\left[f\right]}$ is a local
martingale for any $f\in C$$^{1,2}$ such that $\left.\frac{\partial}{\partial\nu}f(t,\cdot)\right|_{\partial G}=0$
for all $t>0$. Hence, according to the Girsanov theorem, 
\[
M_{t}^{\left[f\right]}-\left\langle N,M^{[f]}\right\rangle _{t}
\]
is a local martingale under the probability $\mathbb{Q}_{0,x}$, where
$N_{t}=\int_{0}^{t}c(r,X_{r})\cdot dW_{r}$. Since the martingale
part $W$ of $X$ is a Brownian motion, so that
\[
\left\langle N,M^{[f]}\right\rangle _{t}=\int_{0}^{t}\left\langle c,\nabla f\right\rangle (r,X_{r})dr
\]
and therefore 
\[
\tilde{M}_{t}^{[f]}=M_{t}^{\left[f\right]}-\int_{0}^{t}\left\langle c,\nabla f\right\rangle (r,X_{r})dr=M_{t}^{\left[f\right]}-\left\langle N,M^{[f]}\right\rangle _{t}
\]
is a local martingale under $\mathbb{Q}_{s,x}$, which completes the
proof. 
\end{proof}
By using Lemma \ref{lem:Lemma2}, for $s<t$ and $x,y\in\bar{G}$
and the fact that both $q_{b}(s,x;t,y)$ and $q_{b+c}(s,x;t,y)$ are
H\"{o}lder continuous, conditional on $X_{t}=y$, we may obtain that
\begin{equation}
\frac{q_{b+c}(s,x;t,y)}{q_{b}(s,x;t,y)}=\mathbb{P}_{s,t}^{x,y}\left[\exp\left\{ \int_{s}^{t}c(r,X_{r})\cdot dW_{r}-\frac{1}{2}\int_{s}^{t}|c|^{2}(r,X_{r})dr\right\} \right],\label{eq:compare-1}
\end{equation}
where $\mathbb{P}_{s,t}^{x,y}$ is the conditional probability $\mathbb{P}_{s,x}\left[\cdot|X_{t}=y\right]$,
which is a probability measure on $(\varOmega,\mathcal{F}_{t})$ given
via the density process
\begin{equation}
\frac{d\mathbb{P}_{s,t}^{x,y}}{d\mathbb{P}_{s,x}}\bigg|_{\mathcal{F}_{r}}=\frac{q_{b+c}(r,X_{r};t,y)}{q_{b}(s,x;t,y)}\quad\forall\ s<r<t.\label{eq:conditional-prob-RN}
\end{equation}
\begin{lem}
\label{lem:Lemma3}Let $b(t,x)$ and $c(t,x)$ be two bounded vector
fields in $\bar{G}$, and assume that $b$ is smooth. Let $(X_{t},\mathbb{P}_{s,x})$
be the reflecting diffusion process with generator $\mathscr{L}_{t,x}$
as in Lemma \ref{lem:Lemma2}. Then
\begin{equation}
q_{b+c}(s,x;T,y)=q_{b}(s,x;T,y)+\int_{s}^{T}\mathbb{P}_{s,x}\left[R_{s,r}c(r,X_{r})\cdot\nabla_{x}q_{b}(r,X_{r};T,y)\right]dr\label{eq:rep-reb1}
\end{equation}
for any $0\leq s<T$, and any $x,y\in\bar{G}$, where $R$ is given
in (\ref{eq:careron-martin-Q}). 
\end{lem}

\begin{proof}
Let $s<T$ and $x,y\in\bar{G}$ be fixed. Then we have two positive
martingales, one is the Cameron-Martin density $R_{t}=R_{s,t}$ given
by (\ref{eq:careron-martin-Q}), which is the exponential martingale
of $N_{t}=\int_{s}^{t}c(r,X_{r})\cdot dW_{r}$, so that
\begin{equation}
R_{t}=1+\int_{s}^{t}R_{r}c(r,X_{r})\cdot dW_{r}\label{eq:des5}
\end{equation}
for $s\leq t\leq T$, which defines the probability $\mathbb{Q}_{s,x}$.
The another is the conditional probability density 
\[
M_{t}=\frac{q_{b}(t,X_{t};T,y)}{q_{b}(s,x;T,y)},\quad\forall\ s<t<T
\]
which determines the conditional probability $\mathbb{P}_{s,T}^{x,y}$,
which can be written as 
\[
M_{t}=\frac{q_{b}(t,X_{t};T,y)}{q_{b}(s,x;T,y)}=e^{\ln q_{b}(t,X_{t};T,y)-\ln q_{b}(s,x;T,y)}.
\]
Since $b$ is smooth, the martingale part of $\ln q_{b}(t,X_{t};T,y)-\ln q_{b}(s,x;T,y)$
equals 
\[
Z_{t}:=\int_{s}^{t}\nabla\ln q_{b}(r,X_{r};T,y)\cdot dW_{r}
\]
so that $M$ must coincide with the exponential martingale of $Z$,
hence
\begin{equation}
M_{t}=1+\int_{s}^{t}M_{r}\nabla\ln q_{b}(r,X_{r};T,y)\cdot dW_{r}\label{eq:des6}
\end{equation}
for $s<t<T$. By (\ref{eq:des5}, \ref{eq:des6}) we have
\[
\left\langle M,R\right\rangle _{t}=\int_{s}^{t}M_{r}R_{r}c(r,X_{r})\cdot\nabla\ln q_{b}(r,X_{r};T,y)dr
\]
and therefore 
\[
M_{t}R_{t}-\left\langle M,R\right\rangle _{t}
\]
is a martingale up to $T$, with $M_{s}R_{s}=1$. Since both $q_{b+c}(s,x;T,y)$
and $q_{b}(s,x;T,y)$ possess the Gaussian bounds (see e.g. \cite{Aronson68,Stroock-SemPaper1988}),
therefore
\begin{align*}
\frac{q_{b+c}(s,x;T,y)}{q_{b}(s,x;T,y)} & =\mathbb{P}_{s,T}^{x,y}\left[R_{T}\right]=\lim_{\varepsilon\downarrow0}\mathbb{P}_{s,T}^{x,y}\left[R_{T-\varepsilon}\right]\\
 & =\lim_{\varepsilon\downarrow0}\mathbb{P}_{s,x}\left[M_{T-\varepsilon}R_{T-\varepsilon}\right]\\
 & =1+\mathbb{P}_{s,x}\left[\int_{s}^{T}M_{r}R_{r}c(r,X_{r})\cdot\nabla\ln q_{b}(r,X_{r};T,y)dr\right]\\
 & =1+\frac{1}{q_{b}(s,x;T,y)}\mathbb{P}_{s,x}\left[\int_{s}^{T}R_{r}c(r,X_{r})\cdot\nabla q_{b}(r,X_{r};T,y)dr\right],
\end{align*}
which completes the proof of the lemma.
\end{proof}
\begin{lem}
\label{lem:Lemma5}Let $\beta$ be a constant and $y\in\bar{G}$.
Let $w(t,x)$ be the unique weak solution to the following non-linear
parabolic equation
\begin{equation}
\frac{\partial}{\partial t}w=\frac{1}{2}\Delta w+\beta|\nabla w|\;\textrm{ for }t>0\textrm{ and }x\in G\label{eq:basiceq1-1}
\end{equation}
subject to the initial and boundary conditions that
\begin{equation}
\left.\frac{\partial}{\partial\nu}w(t,\cdot)\right|_{\partial G}=0\quad\textrm{ for }t>0,\;\textrm{ and }w(0,x)=\delta_{y}(x).\label{eq:bdcod1}
\end{equation}
Then both $w(t,x)$ and its weak derivative $\nabla w(t,x)$ are H\"{o}lder
continuous for $t>0$ and $x\in\bar{G}$, and for any given $T>0$,
\begin{equation}
q_{V}(t,x;T,y)=w(T-t,x)\quad\textrm{ for }0\leq t<T\textrm{ and }x\in\bar{G},\label{eq:ex1}
\end{equation}
where
\[
V(t,x)=\beta\frac{\nabla w(T-t,x)}{\left|\nabla w(T-t,x)\right|}
\]
and $V(t,x)=0$ for $t\geq T$. 
\end{lem}

\begin{proof}
According to the theory of parabolic equations (see e.g. \cite{LadyzhenskajaParabolic1968}),
the problem (\ref{eq:basiceq1-1}, \ref{eq:bdcod1}) has a unique
weak solution $w(t,x)$ which is H\"{o}lder continuous for $t>0$
and $x\in\bar{G}$. We need a bit more regularity of the solution
$w(t,x)$. To this end, for $\varepsilon>0$ consider the semi-linear
parabolic equation
\begin{equation}
\frac{\partial}{\partial t}w^{\varepsilon}=\frac{1}{2}\Delta w^{\varepsilon}+\beta\sqrt{|\nabla w^{\varepsilon}|^{2}+\varepsilon^{2}}\;\textrm{ for }t>0\textrm{ and }x\in G\label{eq:basiceq1-1-1}
\end{equation}
subject to the same initial and boundary conditions (\ref{eq:bdcod1}).
Then, there is a unique strong solution $w^{\varepsilon}(t,x)$ for
every $\varepsilon>0$ which is smooth for $t>0$ and $x\in\bar{G}$.
Let $w_{x}^{\varepsilon}=\nabla w^{\varepsilon}$ denote the space
derivative. By taking derivatives in $x$ for the equation (\ref{eq:basiceq1-1-1}),
we find that $w_{x}^{\varepsilon}=\nabla w^{\varepsilon}$ solves
the Dirichlet boundary problem
\[
\frac{\partial}{\partial t}w_{x}^{\varepsilon}=\left[\frac{1}{2}\Delta+\beta\frac{\nabla w^{\varepsilon}}{\sqrt{\left(\nabla w^{\varepsilon}\right)^{2}+\varepsilon^{2}}}\cdot\nabla\right]w_{x}^{\varepsilon}\;\textrm{ for }t>0\textrm{ and }x\in G
\]
subject to the Dirichlet boundary condition along $\partial G$. Notice
that
\[
\left|\beta\frac{\nabla w^{\varepsilon}}{\sqrt{\left(\nabla w^{\varepsilon}\right)^{2}+\varepsilon^{2}}}\right|\leq|\beta|
\]
is uniformly bounded, so according to Nash's theory (see e.g. \cite{Nash1958},
or \cite{Fabes-Stroock1986,Stroock-SemPaper1988}), there is a convergent
sequence $\left\{ w_{x}^{\varepsilon_{n}}\right\} $ with $\varepsilon_{n}\downarrow0$,
which tends to the weak solution $W$ to the parabolic equation
\[
\frac{\partial}{\partial t}W=\left[\frac{1}{2}\Delta+\beta\frac{\nabla w}{|\nabla w|}\cdot\nabla\right]W
\]
subject to the Dirichlet boundary condition along the boundary $\partial G$
for $t>0$. $W$ is H\"{o}lder continuous in $t>0$ and $x\in G$.
$W$ is a modification of the weak derivative $\nabla w(t,x)$ for
$t>0$ and $x\in G$. We may thus conclude that $\nabla w(t,x)$ is
H\"{o}lder continuous in $(0,\infty)\times G$. 

Given $T>0$, and the unique weak solution $w(t,x)$ to (\ref{eq:basiceq1-1},
\ref{eq:bdcod1}), $u(t,x)=w(T-t,x)$ solves the backward parabolic
equation
\begin{equation}
\frac{\partial}{\partial t}u+\frac{1}{2}\Delta u+\beta\frac{\nabla w(T-t,\cdot)}{\left|\nabla w(T-t,\cdot)\right|}\cdot\nabla u=0\;\textrm{ for }t>0\textrm{ and }x\in G\label{eq:basiceq1-1-2}
\end{equation}
subject to the initial and boundary conditions that
\begin{equation}
\left.\frac{\partial}{\partial\nu}u(t,\cdot)\right|_{\partial G}=0\quad\textrm{ for }t<T,\;\textrm{ and }\lim_{t\uparrow T}u(t,x)=\delta_{y}(x).\label{eq:bdcod1-1}
\end{equation}
Since $q_{V}(s,x;t,y)$ is the fundamental solution of the linear
parabolic equation 
\[
\frac{\partial}{\partial t}u=\frac{1}{2}\Delta u+V(t,x)\cdot\nabla u
\]
subject to the Neumann boundary condition at boundary $\partial G$,
hence, $(t,x)\rightarrow\tilde{u}(t,x)\eqqcolon q_{V}(t,x;T,y)$ solves
the backward equation
\begin{equation}
\frac{\partial}{\partial t}\tilde{u}+\frac{1}{2}\Delta\tilde{u}+\beta\frac{\nabla w(T-t,\cdot)}{\left|\nabla w(T-t,\cdot)\right|}\cdot\nabla\tilde{u}=0\;\textrm{ for }t>0\textrm{ and }x\geq0\label{eq:basiceq1-1-2-1}
\end{equation}
subject to the same initial-boundary conditions (\ref{eq:basiceq1-1-2},
\ref{eq:bdcod1-1}). By the uniqueness, we must have $\tilde{u}(t,x)=u(t,x)$
for $t<T$ and $x\in\bar{G}$. Hence
\[
q_{V}(t,x;T,y)=w(T-t,x)\quad\textrm{ for }t<T\textrm{ and }x\in\bar{G}.
\]
\end{proof}

\subsection*{Proof of Theorem \ref{thm:Theorem1}}

Now we have the major ingredients to prove Theorem \ref{thm:Theorem1}.
Let us explain the ideas leading to the conclusions in Theorem \ref{thm:Theorem1}.
According to the representation formula (\ref{eq:rep-reb1}), it is
apparent that the optimal probability $q_{b}(s,x;T,y)$ is achieved
when 
\[
c(r,x)\cdot\nabla_{x}q_{b}(r,x;T,y)\cdot
\]
has a definite sign for any $c(t,x)$ such that both $|b+c|$ and
$|b|$ are bounded by $\kappa$. Thus for fixed $T>0$ and $y$, we
want to find a vector field $b(t,x)$, which may depend on $T$ and
$y$, such that $|b|\leq\kappa$, and $c(t,x)\cdot\nabla q_{b}(t,x;T,y)$
is non-negative (resp. negative) for all $t<T$ and $x\in\bar{G}$
for all $c(t,x)$ satisfying that $|c+b|\leq\kappa$. Clearly the
best we can do is to choose $b(t,x)$ such that
\[
c(t,x)=A(t,x)\pm\kappa\frac{\nabla q_{b}(t,x;T,y)}{\left|\nabla q_{b}(t,x;T,y)\right|}
\]
where $A(t,x)=c(t,x)+b(t,x)$ so that $|A(t,x)|\leq\kappa$. That
is, the optimal vector fields should satisfy the functional equation
\begin{equation}
b^{\pm}(t,x)=\pm\kappa\frac{\nabla q_{b^{\pm}}(t,x;T,y)}{\left|\nabla q_{b^{\pm}}(t,x;T,y)\right|}\textrm{ for }t\geq0\textrm{ and }x\in\bar{G}.\label{eq:cons1}
\end{equation}
The question becomes to show the existence of such vector fields $b^{\pm}(t,x)$.
Suppose such vector fields exist, then $(t,x)\rightarrow u(t,x):=q_{b^{\pm}}(t,x;T,y)$
is the unique (weak) solution of the Neumann boundary problem to the
backward equation
\begin{equation}
\frac{\partial}{\partial t}u(t,x)+\frac{1}{2}\Delta u(t,x)+b^{\pm}(t,z)\cdot\nabla u(t,x)=0\;\textrm{ for }0<t<T\textrm{ and }x\geq0\label{eq:const2}
\end{equation}
subject to the terminal condition that $\lim_{t\uparrow T}u(t,x)=\delta_{y}(x)$
and the boundary condition that $\left.\frac{\partial}{\partial\nu}u(t,\cdot)\right|_{\partial G}=0$.
Together with (\ref{eq:cons1}), $u(t,x)$ solves the initial and
boundary problem to the semi-linear parabolic equation
\begin{equation}
\frac{\partial}{\partial t}u+\frac{1}{2}\Delta u\pm\kappa|\nabla u|=0\;\textrm{ for }0<t<T\textrm{ and }x\in G\label{eq:basiceq1}
\end{equation}
subject to the initial and boundary conditions above. By the general
theory of parabolic equations, the previous problem (\ref{eq:basiceq1})
has a unique weak solution, see e.g. \cite{LadyzhenskajaParabolic1968}.
The proof is complete. 

\section{Reflecting bang-bang diffusion processes\label{sec:Reflecting-bang-bang-diffusion} }

A closed formula for the solution to the HJB equation (\ref{eq:basiceq1-1},
\ref{eq:bdcod1}) in high dimensions in general is not known. Therefore
let us consider the one dimensional case and $G=[0,\infty)$. For
this case we may work out the explicit formula for the case that $y=0$.
Similar calculations may be carried out for other special domains,
which however must be treated case by case. 

\subsection{Connection with a bang-bang process}

Let $b(t,x),$ defined on $[0,\infty)\times\mathbb{R}^{+}$, be a
bounded, Borel measurable vector field. It is well known that there
is a unique solution to the $\mathscr{L}_{t,x}$-martingale problem
subject to the Neumann boundary condition at $0$, where 
\begin{equation}
\mathscr{L}_{t,x}=\frac{1}{2}\Delta+b(t,x)\cdot\nabla\label{eq:3.1}
\end{equation}
operating on $C^{2}$-functions $f$ on $[0,\infty)$ subject to the
condition that $\frac{\partial f}{\partial x}\rightarrow0$ as $x\downarrow0$. 

The simplest construction of one dimensional reflecting diffusion
processes, due to Skorohod \cite{Skorohod61-62}, is to determine
firstly the diffusion process in the whole line $\mathbb{R}$, that
is the weak solution to the It\^{o} stochastic differential equation
\begin{equation}
dY_{t}=b(t,\left|Y_{t}\right|)\textrm{sgn}(Y_{t})dt+dB_{t},\quad Y_{s}=x.\label{eq:extd1}
\end{equation}
Then for every $x\geq0$, $X_{t}=|Y_{t}|$ is the weak solution to
the following It\^{o}'s stochastic differential equation with boundary
\begin{equation}
dX_{t}=b(t,X_{t})dt+dB_{t}+dL_{t},\quad X_{s}=x,\label{eq:ebm-s1}
\end{equation}
where $t\rightarrow L_{t}$ is continuous and increasing, with initial
zero, and increases only on $\left\{ t\geq0:X_{t}=0\right\} $, so
that $(X_{t})$ is a reflecting diffusion started at $x\geq0$ with
its infinitesimal generator $\mathscr{L}_{t,x}$ together with the
Neumann boundary condition at $0$. Since $\tilde{b}(t,x)=b(t,|x|)\textrm{sgn}(x)$,
which is the odd function extension of $b(t,\cdot)$, is bounded,
according to Aronson \cite{Aronson68} and Nash \cite{Nash1958} (see
e.g. \cite{Fabes-Stroock1986,Norris97,Stroock-SemPaper1988} for simplified
proofs), there is a unique positive and continuous probability density
$p_{\tilde{b}}(s,x;t,y)$ for $t>s\geq0$ and $x,y\in\mathbb{R}$,
which is the heat kernel associated with the elliptic operator $\mathscr{L}_{t,x}=\frac{1}{2}\Delta+\tilde{b}(t,x)\cdot\nabla$,
in the sense that 
\[
\mathbb{E}\left[f(Y_{t})|Y_{s}=x\right]=\int_{\mathbb{R}}p_{\tilde{b}}(s,x;t,y)f(y)dy
\]
for positive or bounded Borel measurable function $f$. In fact $p_{\tilde{b}}(s,x;t,y)$
is the fundamental solution (in the weak solution sense) to the linear
parabolic equation
\[
\left(\frac{\partial}{\partial s}+\frac{1}{2}\Delta+\tilde{b}(s,\cdot)\nabla\right)u(s,x)=0
\]
for $s\geq0$ and $x\in\mathbb{R}$. $p_{\tilde{b}}(s,x;t,y)$ is
bounded from above and below by Gaussian functions (see e.g. \cite{Aronson68,Norris97}
for a precise statement), and is H\"{o}lder continuous in $s<t$
and $x,y\in\mathbb{R}$. As a consequence of Skorohod's construction,
the reflecting diffusion $(X_{t})$ possesses a continuous transition
probability density denoted by $q_{b}(s,x;t,y)$ (for $s<t$ and $x\geq0$,
$y\geq0$), that is, 
\[
\mathbb{E}\left[f(X_{t})|X_{s}=x\right]=\int_{[0,\infty)}q_{b}(s,x;t,y)f(y)dy,
\]
and
\begin{equation}
q_{b}(s,x;t,y)=p_{\tilde{b}}(s,x;t,y)+p_{\tilde{b}}(s,x;t,-y)\label{eq:ref-des1}
\end{equation}
for any $0\leq s<t$ and $x\geq0$, $y\geq0$. 

If $|b(t,x)|\leq\kappa$ for all $t\geq0$ and $x\geq0$, then $|\tilde{b}(t,x)|\leq\kappa$,
by applying Theorem 1 of \cite{QZ02} together with (\ref{eq:ref-des1})
we have the following corollary. 
\begin{cor}
\label{cor:bounds}If $|b(t,x)|\leq\kappa$ for $t>0$ and $x\geq0$,
then the transition probability density $q_{b}(s,x;t,y)$ of the reflecting
diffusion $(X_{t})$ possesses the following bounds 
\begin{equation}
\begin{aligned} & p_{y}^{-\kappa}(x,t-s,y)+p_{-y}^{-\kappa}(x,t-s,-y)\leq q_{b}(s,x;t,y)\\
 & \quad\quad\quad\leq p_{y}^{\kappa}(x,t-s,y)+p_{-y}^{\kappa}(x,t-s,-y),
\end{aligned}
\label{eq:QZ-bounds}
\end{equation}
for all $0\leq s<t$ and any $x,y\geq0$, where $p_{y}^{\beta}(x,t,z)$
is the transition probability density function of the diffusion process
\begin{equation}
dZ_{t}=-\beta\mathrm{sgn}(Z_{t}-y)dt+dB_{t}\label{eq:babg01}
\end{equation}
so that
\[
p_{y}^{\beta}(x,t,y)=\frac{1}{\sqrt{2\pi t}}\int_{|x-y|/\sqrt{t}}^{\infty}ze^{-(z-\beta\sqrt{t})^{2}/2}dz.
\]
In the case $y=0$, the bounds in (\ref{eq:QZ-bounds}) are optimal.
\end{cor}

\begin{proof}
Let us show that the bounds in (\ref{eq:QZ-bounds}) are optimal if
$y=0$. To this end we consider the reflecting diffusion $(X_{t})$
in $[0,\infty)$ with a linear drift, i.e. the weak solution to
\begin{equation}
dX_{t}=dB_{t}+\beta dt+dL_{t}\label{eq:reflected-bangbang}
\end{equation}
where $L_{t}$ increases only when $X$ hits zero, whose transition
probability $q_{\beta}(s,x;t,z)$ is time homogeneous. The corresponding
diffusion process $Y$ in the Skorohod construction, so that $X=|Y|$,
is the weak solution to the stochastic differential equation
\begin{equation}
dY_{t}=dB_{t}+\beta\mathrm{sgn}(Y_{t})dt
\end{equation}
which is the special case of the bang-bang process whose transition
probability is $p^{\beta}(x,t,z)$ and therefore 
\begin{equation}
q_{\beta}(s,x;t,z)=p^{\beta}(x,t-s,z)+p^{\beta}(x,t-s,-z)\label{eq:density1}
\end{equation}
for $x\geq0$ and $z\geq0$. The transition probability density $p^{\beta}(x,t,z)$
can be worked out by using Cameron-Martin formula as in \cite{KS84,KS91,QZ02},
which is given by
\begin{equation}
\begin{aligned}p^{\beta}(x,t,z) & =\frac{1}{\sqrt{2\pi t}}e^{-\frac{1}{2t}[(x-z)^{2}-2\beta t(|z|-|x|)+\beta^{2}t^{2}]}\\
 & \quad-\beta e^{2\beta|z|}\int_{|x|+|z|+\beta t}^{+\infty}\frac{1}{\sqrt{2\pi t}}e^{-\frac{u^{2}}{2t}}du,
\end{aligned}
\label{eq:density2}
\end{equation}
for any $x,z\in\mathbb{R}$. In particular
\begin{equation}
\nabla_{x}q_{\beta}(t,x,T,y)=-\frac{1}{\sqrt{2\pi(T-t)^{3}}}e^{-\frac{(x-y+\beta(T-t))^{2}}{2(T-t)}}\left[x-y+\beta(T-t)+e^{-\frac{2xy}{T-t}}(x+y-\beta(T-t))\right]\label{eq:de-sign}
\end{equation}
for $x,y\geq0$ and $t<T$. In general, if $y>0$, then $\nabla_{x}q_{\beta}(t,x;T,y)$
has a zero $x>0$ and thus changes its sign. While, if $y=0$, then
\[
q_{\beta}(t,x;T,0)=\frac{2}{\sqrt{2\pi t}}\int_{x/\sqrt{t}}^{\infty}ze^{-(z+\beta\sqrt{t})^{2}/2}dz
\]
for $x\geq0$, so that $\nabla q_{\beta}(t,x;T,0)\leq0$, and thus
\[
-\beta\textrm{sgn}\left(q_{\beta}(t,x;T,0)\right)=\beta
\]
for $x\geq0$. Hence, according to Theorem \ref{thm:Theorem1}, for
any $T>0$ and $y=0$, the corresponding vector fields which optimize
$q_{b}(t,x;T,y)$ (where $|b|\leq\kappa$) are constants $b^{\pm}(t,x)=\mp\kappa$.
Therefore the bounds in (\ref{eq:QZ-bounds}) are optimal when $y=0$. 
\end{proof}

\subsection{A reflecting bang-bang process}

When $G=(-\infty,\infty)$, then there is no reflection, the optimal
bounds are attained by the bang-bang processes (\ref{eq:babg01}).
One then would wonder, given  $T>0$ and $y>0$, whether the optimal
probability $q_{b}(t,x;T,y)$ also should be attained by the reflecting
diffusion processes of bang-bang processes, that is, the diffusion
processes obtained by solving stochastic differential equation in
$[0,\infty)$ with boundary $0$:
\begin{equation}
dX_{t}=-\beta\textrm{sgn}(X_{t}-y)dt+dB_{t}+L_{t}.\label{eq:bang-bang-reflecting}
\end{equation}
In the case that $y>0$, the sign of $X_{t}-y$ cannot be determined
even though $X_{t}\geq0$. In order to calculate its transition density
function, which is time homogeneous, denoted by $q(t,x,y)$ for simplicity,
and to determine the sign of $\frac{\partial}{\partial x}q(t,x,y)$,
one needs to compute the probability density $p(t,x,y)$ to the associated
bang-bang process
\begin{equation}
dY_{t}=-\beta\textrm{sgn}(Y_{t})\textrm{sgn}(|Y_{t}|-y)dt+dB_{t},\label{eq:bang-bang-bang}
\end{equation}
which in turn requires the joint distribution of Brownian motion and
local times of Brownian motion at three distinct points. 

It is interesting by its own for calculating the transition probability
density $p(t,x,y)$ for the bang-bang process with three singularities.
Let $(B_{t},\mathbb{P}_{x})$ be standard Brownian motion on $(\Omega,\mathcal{F})$.
Consider the one dimensional diffusion process $\left\{ \mathbb{Q}_{x}:\ x\in\mathbb{R}\right\} $
associated with the generator $\mathscr{L}=\frac{1}{2}\Delta+b(x)\cdot\nabla$,
where $b(x)=-\beta\textrm{sgn}(|x|-y)\textrm{sgn}(x)$ and $y>0$.
For this case, the Cameron-Martin density for $0\leq s<t$ is defined
by 
\[
R_{t}=\exp\left[\int_{0}^{t}-\beta\textrm{sgn}(|B_{r}|-y)\textrm{sgn}(B_{r})dB_{r}-\frac{1}{2}\beta^{2}t\right],
\]
and therefore $R_{t}$ is the Radon-Nikodym derivative of $\mathbb{Q}_{x}$
with respect to the Wiener measure $\mathbb{P}_{x}$ restricted over
$(\varOmega,\mathcal{F}_{t})$, where $\mathcal{F}_{t}=\sigma(\{B_{s}:\ s\leq t\})$.
Notice that, for $t>0$ and $x,z\in\mathbb{R}$, we have 
\[
\frac{p(t,x,z)}{h(t,x,z)}=\mathbb{P}_{t}^{x,z}\left\{ \exp\left[\int_{0}^{t}-\beta\textrm{sgn}(|B_{r}|-y)\textrm{sgn}(B_{r})dB_{r}-\frac{1}{2}\beta^{2}t\right]\right\} ,
\]
where $h(t,x,z)=\frac{1}{\sqrt{2\pi t}}e^{-\frac{(x-z)^{2}}{2t}}$
is the heat kernel, and $\mathbb{P}_{t}^{x,z}$ is the Brownian motion
bridge measure. For $\varepsilon>0$ small we have 
\[
\left.\frac{d\mathbb{P}_{t}^{x,z}}{d\mathbb{P}_{x}}\right|_{\mathcal{F}_{t-\varepsilon}}=\frac{h(\varepsilon,B_{t-\varepsilon},z)}{h(t,x,z)},
\]
so that 
\begin{equation}
p(t,x,z)=\lim_{\varepsilon\downarrow0}\mathbb{P}_{x}\left\{ h(\varepsilon,B_{t-\varepsilon},z)\exp\left[\int_{0}^{t-\varepsilon}-\beta\textrm{sgn}(|B_{r}|-y)\textrm{sgn}(B_{r})dB_{r}-\frac{1}{2}\beta^{2}t\right]\right\} .
\end{equation}
Let 
\begin{equation}
\phi_{y}(x)=||x|-y|.\label{eq:phi-yx}
\end{equation}
Then by It\^{o}-Tanaka formula, 
\begin{equation}
\phi_{y}(B_{t})=\phi_{y}(x)+\int_{0}^{t}\textrm{sgn}(|B_{r}|-y)\textrm{sgn}(B_{r})dB_{r}+L_{t}^{y}-L_{t}^{0}+L_{t}^{-y},
\end{equation}
where $L_{t}^{a}$ is the local time of $B_{t}$ at $a$, and
\begin{equation}
p(t,x,z)=\lim_{\varepsilon\downarrow0}\mathbb{P}_{x}\left\{ h(\varepsilon,B_{t-\varepsilon},z)\exp\left[-\beta\left(\phi_{y}(B_{t-\varepsilon})-\phi_{y}(x)+L_{t-\varepsilon}^{0}-L_{t-\varepsilon}^{y}-L_{t-\varepsilon}^{-y}\right)-\frac{1}{2}\beta^{2}t\right]\right\} .\label{eq:ptxz}
\end{equation}
Let $f_{x,y,t}(u,w)$ be the density of the joint distribution of
$(L_{t}^{0}-L_{t}^{y}-L_{t}^{-y},B_{t})$, that is,

\[
\begin{aligned}\mathbb{P}_{x}(L_{t}^{0}-L_{t}^{y}-L_{t}^{-y}\in du,B_{t}\in dw) & =f_{x,y,t}(u,w)dudw.\end{aligned}
\]
Then 
\begin{equation}
\begin{aligned} & \quad p(t,x,z)\\
 & =\lim_{\varepsilon\downarrow0}\iint_{u,w\in\mathbb{R}}h(\varepsilon,w,z)e^{-\beta\left(\phi_{y}(w)-\phi_{y}(x)+u\right)-\frac{1}{2}\beta^{2}t}f_{x,y,t-\varepsilon}(u,w)dudw\\
 & =\int_{-\infty}^{\infty}e^{-\beta\left(\phi_{y}(z)-\phi_{y}(x)+u\right)-\frac{1}{2}\beta^{2}t}f_{x,y,t}(u,z)du\\
 & =e^{-\beta\left(\phi_{y}(z)-\phi_{y}(x)\right)-\frac{1}{2}\beta^{2}t}\int_{-\infty}^{\infty}e^{-\beta u}f_{x,y,t}(u,z)du.
\end{aligned}
\label{eq:q-yk-1}
\end{equation}
Therefore, the transition probability density function 
\begin{align}
q(t,x,y) & =p(t,x,y)+p(t,x,-y)\nonumber \\
 & =e^{\beta|x-y|-\frac{1}{2}\beta^{2}t}\int_{-\infty}^{\infty}e^{-\beta u}\left[f_{x,y,t}(u,y)+f_{x,y,t}(u,-y)\right]du.
\end{align}
The joint distribution of $(L_{t}^{0}-L_{t}^{y}-L_{t}^{-y},B_{t})$
or $(L_{t}^{-y},L_{t}^{0},L_{t}^{y},B_{t})$ is, however, not known.
Here, we give another strategy to compute the transition probability
density function $p(t,x,z)$. That is, we first compute the expectation
(\ref{eq:ptxz}) at a random time $\tau$, where $\tau$ is a random
variable independent of the Brownian motion $B_{t}$ and has the exponential
distribution $\mathbb{P}(\tau>t)=e^{-\lambda t}$ for $t\geq0$ and
$\lambda>0$. The motivation for computation at a random time $\tau$
is that one can get the solution by solving an ordinary differential
equation rather than a partial differential equation. Similar ideas
have been used, for example, in \cite{Boro-Salm02,Mansuy-Yoy2008}
for calculating various distributions of Brownian functionals. By
applying inverse Laplace transformation in time $t$, we may obtain
$p(t,x,z)$ at a fixed time $t$, since formally 
\begin{align}
p_{y,\lambda}^{\beta}(x,z) & :=\mathbb{P}_{x}\left\{ \mathds{1}_{B_{\tau}=z}\exp\left[-\beta\left(\phi_{y}(B_{\tau})-\phi_{y}(x)+L_{\tau}^{0}-L_{\tau}^{y}-L_{\tau}^{-y}\right)-\frac{1}{2}\beta^{2}\tau\right]\right\} \\
 & =\int_{0}^{\infty}\lambda e^{-\lambda t}\mathbb{P}_{x}\left\{ \mathds{1}_{B_{t}=z}\exp\left[-\beta\left(\phi_{y}(B_{t})-\phi_{y}(x)+L_{t}^{0}-L_{t}^{y}-L_{t}^{-y}\right)-\frac{1}{2}\beta^{2}t\right]\right\} dt\nonumber \\
 & =\int_{0}^{\infty}\lambda e^{-\lambda t}p(t,x,z)dt.
\end{align}
So we may define the Laplace transformation $U(x):=\lambda^{-1}p_{y,\lambda}^{\beta}(x,z)$
of $p(t,x,z)$, then 
\begin{equation}
\frac{1}{2}U''(x)+b(x)U'(x)-\lambda U(x)=-\delta_{z}(x),\quad x\in\mathbb{R},\label{eq:ODE-1}
\end{equation}
where \textbf{$b(x)=-\beta\mathrm{sgn}(|x|-y)\mathrm{sgn}(x)$} and
$y>0$. Besides, we know that $U(x)$ is continuous, and satisfies
\begin{equation}
\lim_{x\to+\infty}U(x)=0,\quad\lim_{x\to-\infty}U(x)=0.\label{eq:ODE-2}
\end{equation}
Sometimes we denote $U(x)=U_{z}(x)$ to emphasize the dependence on
$z$. Alternately we may directly compute the Laplace transformation
$V(x)=V_{z}(x)$ of the transition probability density $q(t,x,z)$,
which satisfies the ordinary differential equation: 
\begin{equation}
\begin{cases}
\frac{1}{2}V''(x)+\beta\mathrm{sgn}(y-x)V'(x)-\lambda V(x)=-\delta_{z}(x), & x>0\\
V'(0+)=0.
\end{cases}
\end{equation}
Solving the above equations, we obtain for any $x,y\geq0$,
\begin{equation}
V_{y}(x)=U_{y}(x)+U_{-y}(x)=\begin{cases}
C_{1}e^{-(\beta+\sqrt{\beta^{2}+2\lambda})x}+C_{2}e^{-(\beta-\sqrt{\beta^{2}+2\lambda})x}, & 0\leq x\leq y,\\
C_{3}e^{(\beta-\sqrt{\beta^{2}+2\lambda})x}, & x\geq y,
\end{cases}\label{eq:Laplace-r1}
\end{equation}
where 
\begin{align}
C_{1} & =\left[(\beta+\bar{\beta})e^{-(\beta-\bar{\beta})y}-\beta e^{-(\beta+\bar{\beta})y}\right]^{-1},\label{eq:Laplace-r11}\\
C_{2} & =\frac{\beta+\bar{\beta}}{2\lambda e^{-(\beta-\bar{\beta})y}-(\bar{\beta}-\beta)\beta e^{-(\beta+\bar{\beta})y}},\label{eq:Laplace-r12}\\
C_{3} & =\frac{(\bar{\beta}-\beta)e^{-2\beta y}+(\bar{\beta}+\beta)e^{-2(\beta-\bar{\beta})y}}{2\lambda e^{-(\beta-\bar{\beta})y}-(\bar{\beta}-\beta)\beta e^{-(\beta+\bar{\beta})y}},\label{eq:Laplace-r13}
\end{align}
and $\bar{\beta}=\sqrt{\beta^{2}+2\lambda}$.

If the Laplace transformation is $F(\lambda)=\int_{0}^{\infty}e^{-\lambda t}f(t)dt$,
the inverse Laplace transformation is denoted by 
\[
\mathcal{L}_{\lambda}^{-1}(F(\lambda))=:f(t).
\]
Then, for any $x,y\geq0$, the transition probability density $q(t,x,y)$
of the reflected diffusion (\ref{eq:bang-bang-reflecting}) is then
the inverse Laplace transformation: 
\begin{equation}
q(t,x,y)=\mathcal{L}_{\lambda}^{-1}\left(V_{y}(x)\right).\label{eq:inverse-ptxy}
\end{equation}
So we may conclude the above computations as the following theorem. 
\begin{thm}
The Laplace transformation of the transition probability density $q(t,x,y)$
of the reflected diffusion (\ref{eq:bang-bang-reflecting}) is the
function $V_{y}(x)$ in (\ref{eq:Laplace-r1}) with coefficients (\ref{eq:Laplace-r11})-(\ref{eq:Laplace-r13}). 
\end{thm}

Even though it is not easy to work out a closed analytic form of the
transition probability density $q(t,x,y)$ of the reflecting diffusion
(\ref{eq:bang-bang-reflecting}), by the numerical method for the
computation of inverse Laplace transformation, see for example \cite{AW06},
we can get the precise value of the transition probability density
$q(t,x,y)$ for any $\beta\in\mathbb{R}$, $t>0$ and $x,y\geq0$.
The numerical test for (\ref{eq:inverse-ptxy}) reveals that the reflecting
bang-bang diffusion processes (\ref{eq:bang-bang-reflecting}) are
not the optimal diffusion process except $y=0$. 

\section{The HJB equation-One dimensional case\label{sec:HJB}}

The solution $w(t,x)$ to the HJB equation (with reflecting boundary)
(\ref{eq:basiceq1-1}, \ref{eq:bdcod1}) plays the dominated role
in our discussion, thus it is interesting to look for its properties
in order to gain further knowledge about the optimal probability $q_{b}(t,x;T,y)$
where $|b|\leq\kappa$. We still consider the case where $G=[0,\infty)$.
The solution for the case where $y=0$ has been obtained in the previous
section. Therefore, in this section we assume that $y>0$. 

Let $\beta(=\pm\kappa)$ be a constant. Recall that, for one dimensional
case with $G=[0,\infty)$, the HJB equation for our optimization problem
is the boundary problem 

\begin{equation}
\frac{\partial}{\partial t}w=\frac{1}{2}\Delta w+\beta|\nabla w|\;\textrm{ for }t>0\textrm{ and }x\geq0\label{eq:basiceq1-1-3}
\end{equation}
subject to the initial and boundary conditions that
\begin{equation}
\lim_{x\downarrow0}\frac{\partial}{\partial x}w(t,x)=0\quad\textrm{ for }t>0,\;\textrm{ and }w(0,x)=\delta_{y}(x).\label{eq:bdcod1-2}
\end{equation}
The solution $w(t,x)>0$ for all $t>0$ and $x\geq0$ by the maximal
principle and $w_{x}(t,x)=\frac{\partial}{\partial x}w(t,x)$ (for
$t>0$ and $x\geq0$) is H\"{o}lder continuous in $t>0$ and $x\geq0$. 

To gain more explicit information about the optimal bounds in (\ref{eq:com1}),
we need to understand the space derivative $\frac{\partial}{\partial x}w(t,x)$.
For $t=\tau>0$ is sufficiently small 
\[
w(\tau,x)\cong\frac{1}{\sqrt{2\pi\tau}}\left\{ e^{-\frac{(x-y)^{2}}{2\tau}}+e^{-\frac{(x+y)^{2}}{2\tau}}\right\} 
\]
and
\[
w_{x}(\tau,x)\cong-\frac{1}{\sqrt{2\pi\tau^{3}}}e^{-\frac{(x-y)^{2}}{2\tau}}\left\{ x-y+(x+y)e^{-\frac{2xy}{\tau}}\right\} 
\]
which implies that for $\tau>0$ small enough, $w_{x}$ has exactly
one zero near $y$ other than $0$, denoted by $s(\tau)>0$.

We have plotted the figures of the derivative $\nabla w(t,x)$ for
fixed $\text{\ensuremath{\beta}}=1$ and $y=0,1,5,10$, respectively,
and $t\in[0.5,5]$ and $x\in[0,15]$ in the Figure \ref{fig_derivative}.
Figure 1 shows, as long as $y>0$, there is at most one root other
than $0$ to the equation $w_{x}(t,x)=0$ for every $t>0$.  For
$y>0$, there exists $\tau=\tau_{y,\beta}>0$, such that there is
exactly one $s(t)>0$ for every $0<t<\tau_{y,\beta}$ such that $w_{x}(t,s(t))=0$,
and for every $t\geq\tau_{y,\beta}$ there is no zero of $w_{x}(t,\cdot)$,
i.e. $w_{x}(t,x)<0$, for any $x>0$. In Figure \ref{fig_free_boundary},
we have plotted the zeros $s(t)$ for fixed $y>0$ and $\beta=1$.
The point which $s(t)$ crosses $t$-axis is the time $\tau_{y,\beta}$.
So the initial and boundary problem (\ref{eq:basiceq1-1-3}, \ref{eq:bdcod1-2})
may be equivalent to a free boundary problem.

\begin{figure}[!htbp]
\begin{centering}
\begin{minipage}[c]{7.8cm}%
 \includegraphics[width=8.2cm,height=7cm]{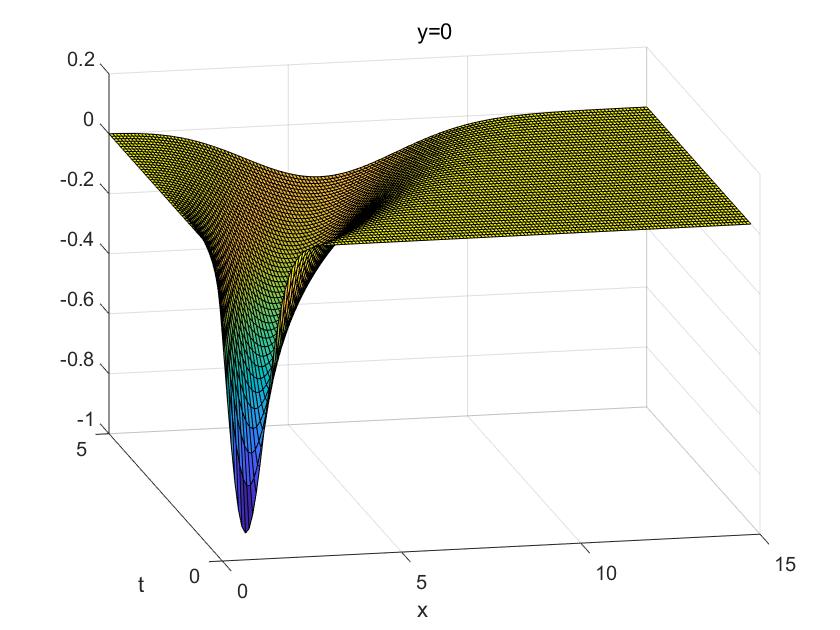} %
\end{minipage}%
\begin{minipage}[c]{7.8cm}%
 \includegraphics[width=8.2cm,height=7cm]{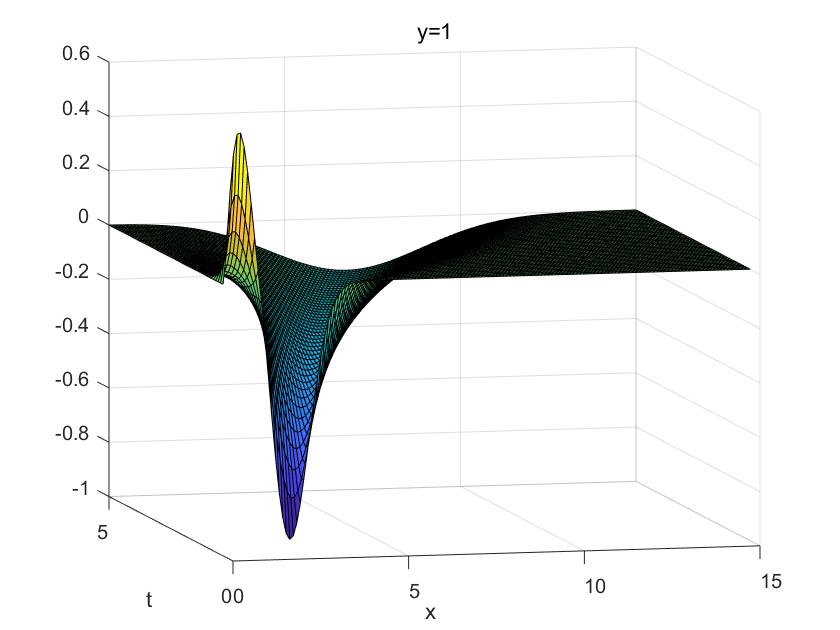} %
\end{minipage}
\par\end{centering}
\begin{centering}
\begin{minipage}[c]{7.8cm}%
 \includegraphics[width=8.2cm,height=7cm]{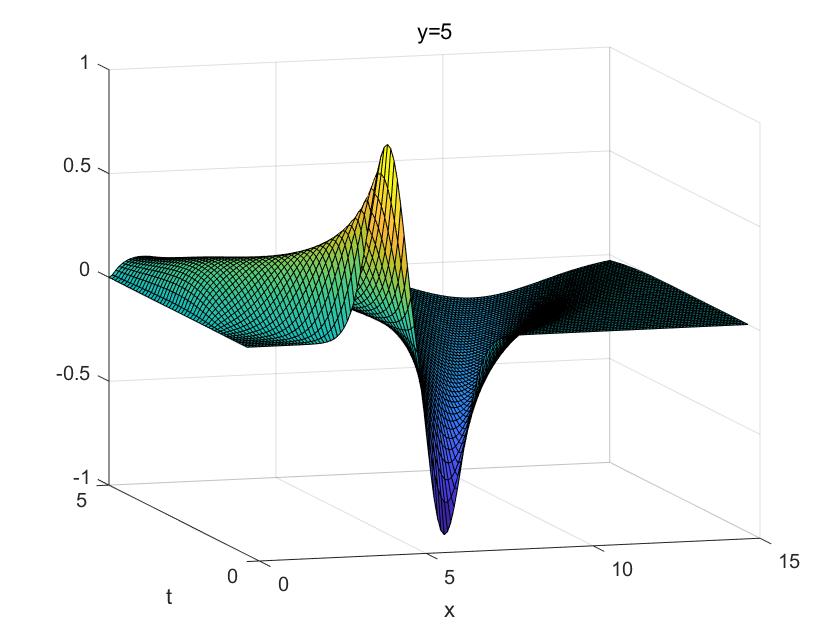} %
\end{minipage}%
\begin{minipage}[c]{7.8cm}%
 \includegraphics[width=8.2cm,height=7cm]{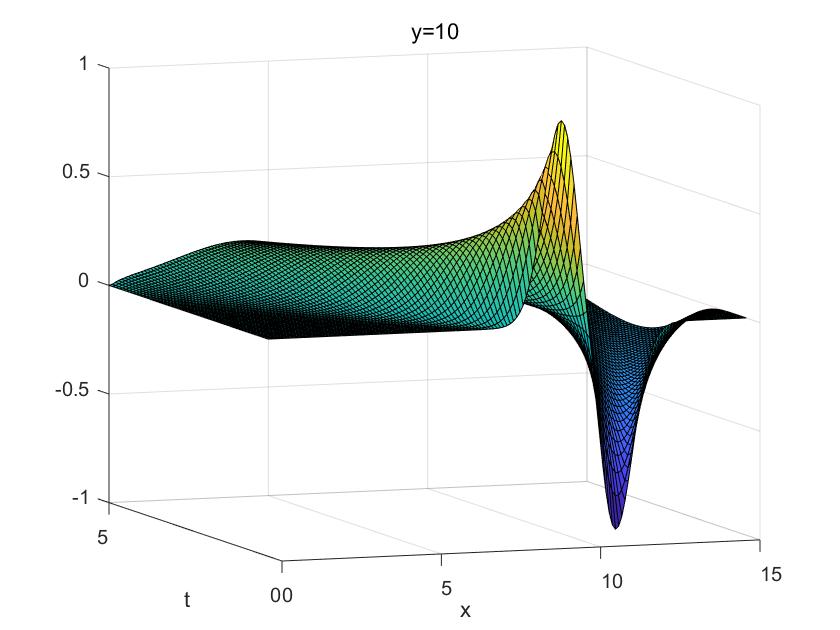} %
\end{minipage}
\par\end{centering}
\caption{Derivative $\nabla w(t,x)$}
\label{fig_derivative} 
\end{figure}

\begin{figure}[!htbp]
\begin{centering}
\begin{minipage}[c]{7.8cm}%
 \includegraphics[width=8.2cm,height=7cm,keepaspectratio]{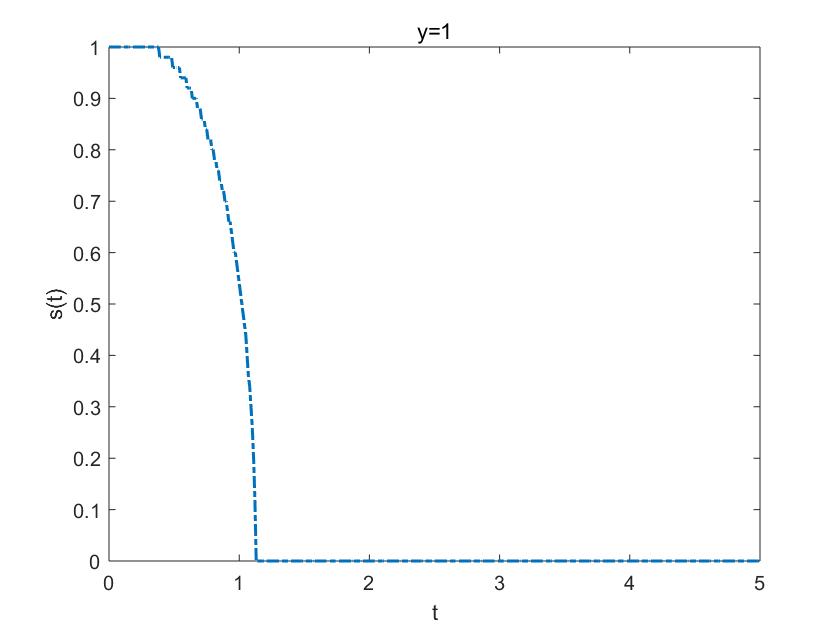} %
\end{minipage}%
\begin{minipage}[c]{7.8cm}%
 \includegraphics[width=8.2cm,height=7cm,keepaspectratio]{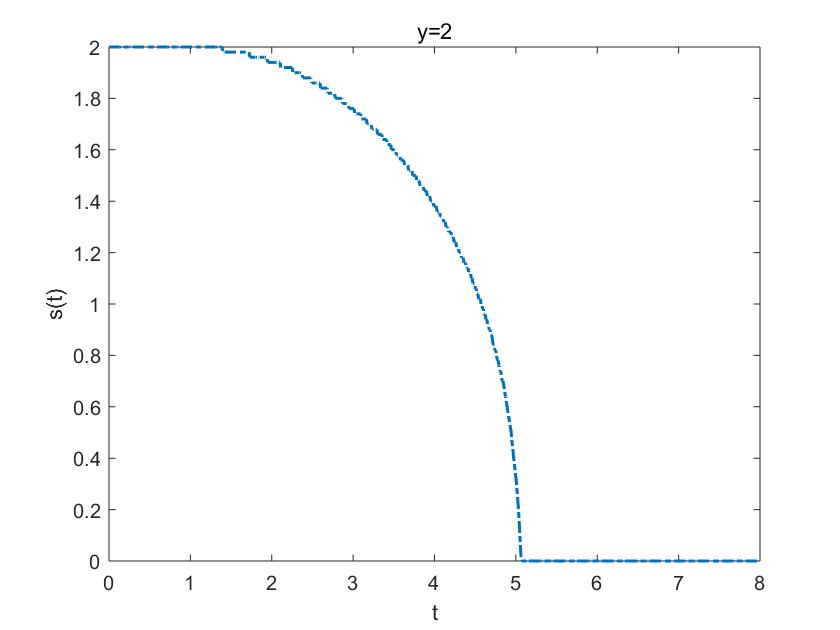} %
\end{minipage}
\par\end{centering}
\begin{centering}
\begin{minipage}[c]{7.8cm}%
 \includegraphics[width=8.2cm,height=7cm,keepaspectratio]{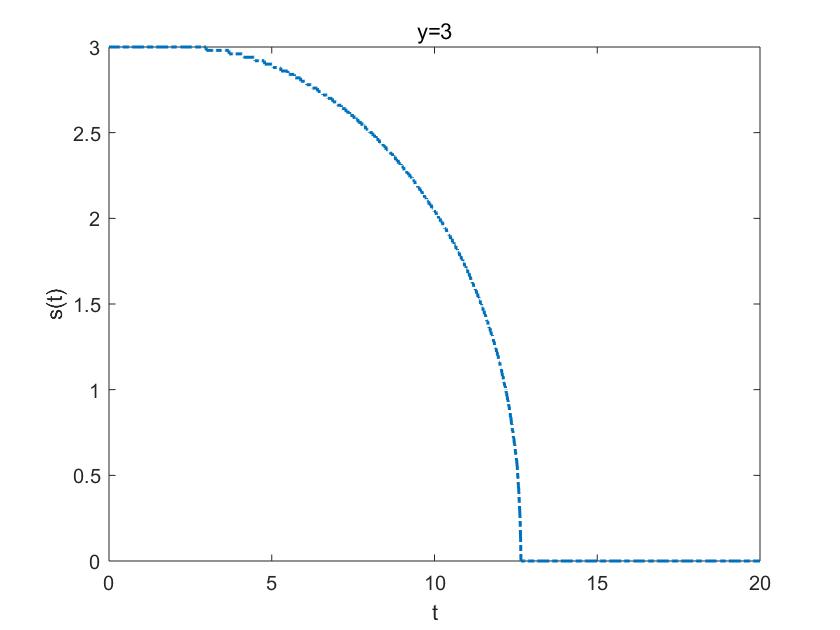} %
\end{minipage}%
\begin{minipage}[c]{7.8cm}%
 \includegraphics[width=8.2cm,height=7cm,keepaspectratio]{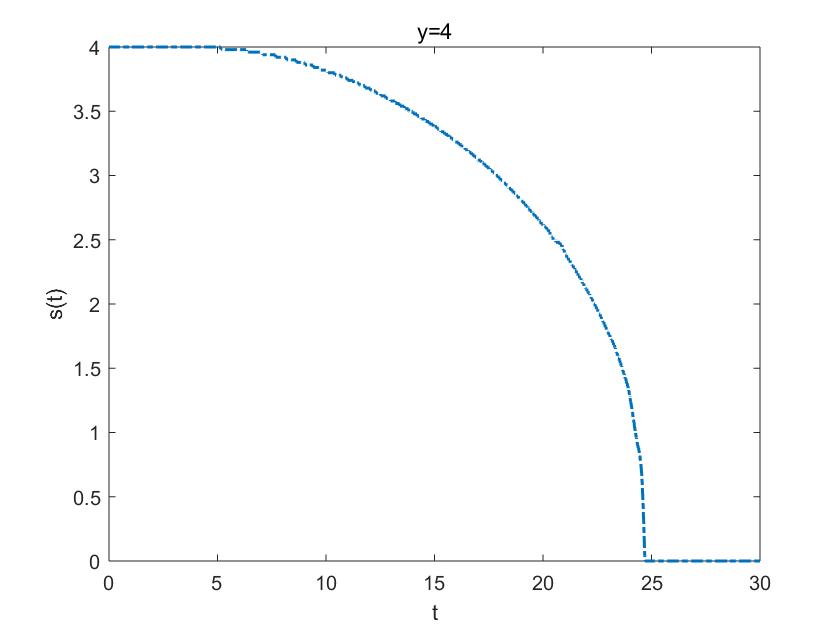} %
\end{minipage}
\par\end{centering}
\caption{Free boundary $s(t)$ for fixed $y>0$ demonstrating feature of ``phase
transition''}
\label{fig_free_boundary} 
\end{figure}

\newpage{}

\section{Application in Stochastic Optimal Control\label{sec:Stochastic-Control} }

In this section, we consider a stochastic optimal control problem
related to reflecting diffusion processes. Let 
\begin{equation}
X_{t}=x+W_{t}+\int_{0}^{t}u_{s}ds+L_{t}^{u}
\end{equation}
be a diffusion type process reflecting at zero, where $u$ is adapted
and satisfies $|u|_{\infty}\leq\kappa$ on the time interval $\mathbb{R}$.
We denote all these controls $u$ as an admissible set $\mathscr{U}$.
One problem is to minimize the cost functional 
\[
J(u)=\mathbb{E}_{x}\left[\int_{0}^{T}f(t,X_{t},u_{t})dt+h(X_{T})\right]
\]
by choosing an optimal $u\in\mathscr{U}$. Our interest in this paper
is to minimize the following expected discounted cost with infinite
horizon: 
\[
J(u)=\mathbb{E}_{x}\int_{0}^{\infty}e^{-\lambda t}f(X_{t})dt,
\]
where we take $T=\infty$, and $h=0$. The problem has been studied
in e.g. \cite{BSW80,KS84,KS91,Shreve81} for diffusion processes with
different constraints. Here we consider the case with the reflecting
boundary conditions.
\begin{thm}
\label{thm:stochastic-control}Let $f(x)$ be of at most polynomial
growth, and let 
\begin{equation}
v(x)=\inf_{u\in\mathscr{U}}\mathbb{E}_{x}\int_{0}^{\infty}e^{-\lambda t}f(X_{t})dt.
\end{equation}
Then $v(x)$ is the solution to the ordinary differential equation
\begin{align}
\frac{1}{2}v''+f(x) & =\kappa|v'|+\lambda v,\label{eq:control-v1}\\
v'(0+) & =0,\label{eq:control-v2}
\end{align}
on $[0,\infty)$, with at most polynomial growth when $x$ is large
enough. 
\end{thm}

\begin{proof}
The equations (\ref{eq:control-v1}) and (\ref{eq:control-v2}), together
with the polynomial growth at infinity, has a unique classical solution
$v(x)\in C^{2}([0,\infty))$. Let $V_{u}(x)=\mathbb{E}_{x}\int_{0}^{\infty}e^{-\lambda t}f(X_{t})dt$,
we will show that the solution $v(x)=\inf_{u\in\mathscr{U}}V_{u}(x)$.
Define the process 
\begin{equation}
M_{t}=e^{-\lambda t}v(X_{t})+\int_{0}^{t}e^{-\lambda s}f(X_{s})ds.
\end{equation}
By Itô formula, we have 
\begin{align*}
M_{t} & =M_{s}+\int_{s}^{t}e^{-\lambda r}\left(-\lambda v(X_{r})+u_{r}v'(X_{r})+\frac{1}{2}v''(X_{r})+f(X_{r})\right)dr\\
 & \quad\quad+\int_{s}^{t}e^{-\lambda r}v'(X_{r})dW_{r}+\int_{s}^{t}e^{-\lambda r}v'(X_{r})dL_{r}^{u},
\end{align*}
for any $s<t$. Since 
\[
\begin{aligned} & \quad-\lambda v+u_{r}v'+\frac{1}{2}v''+f\\
 & \geq-\lambda v+\inf_{u\in\mathscr{U}}(u_{r}v')+\frac{1}{2}v''+f\\
 & =-\lambda v-\kappa|v'|+\frac{1}{2}v''+f=0,
\end{aligned}
\]
and the support of $L_{\cdot}^{u}$ is $\{t\geq0:\ X_{t}=0\}$ a.s.,
and $v'(0+)=0$, so we have 
\[
M_{t}\geq M_{s}+\int_{s}^{t}e^{-\lambda r}v'(X_{r})dW_{r}.
\]
Thus, 
\begin{equation}
\mathbb{E}_{x}[M_{t}|\mathcal{F}_{s}]\geq M_{s},\quad\textrm{for}\ \forall\ s\leq t.
\end{equation}
That is, $M_{t}$ is a submartingale. So 
\[
\mathbb{E}_{x}M_{t}=e^{-\lambda t}\mathbb{E}_{x}v(X_{t})+\mathbb{E}_{x}\int_{0}^{t}e^{-\lambda s}f(X_{s})ds\geq M_{0}=v(x).
\]
Let $t\to\infty$, then 
\begin{equation}
\mathbb{E}_{x}\int_{0}^{\infty}e^{-\lambda t}f(X_{t})dt\geq v(x),\quad\textrm{for all}\ u\in\mathscr{U}.
\end{equation}
On the other hand, by taking 
\begin{equation}
u_{t}^{*}=-\kappa\mathrm{sgn}(v'(X_{t}))\in\mathscr{U},
\end{equation}
similarly we know that $M_{t}$ is a martingale and $\mathbb{E}_{x}M_{t}=v(x)$
for any $t\geq0$. So 
\begin{equation}
v(x)=\mathbb{E}_{x}\int_{0}^{\infty}e^{-\lambda t}f(X_{t}^{*})dt\geq\inf_{u\in\mathscr{U}}\mathbb{E}_{x}\int_{0}^{\infty}e^{-\lambda t}f(X_{t})dt,
\end{equation}
where 
\[
X_{t}^{*}=x+W_{t}+\int_{0}^{t}u_{s}^{*}ds+L_{t}^{u^{*}}.
\]
Therefore, we have completed the proof. Besides, we also know that
$u^{*}$ is the optimal stochastic control for our problem. 
\end{proof}
In fact we may obtain the explicit solution for the stochastic optimal
control problem by using some simple algebra for the cases where $f(x)=x$
and $f(x)=x^{2}$. 

If $f(x)=x$, then we have the value function
\begin{equation}
v(x)=\frac{e^{(\kappa-\sqrt{\kappa^{2}+2\lambda})x}}{\lambda(-\kappa+\sqrt{\kappa^{2}+2\lambda})}+\frac{x}{\lambda}-\frac{\kappa}{\lambda^{2}},\quad\textrm{on}\ [0,\infty).
\end{equation}

If $f(x)=x^{2}$, the value function $v(x)$ is 
\begin{equation}
v(x)=\frac{2\kappa e^{(\kappa-\sqrt{\kappa^{2}+2\lambda})x}}{\lambda^{2}(\kappa-\sqrt{\kappa^{2}+2\lambda})}+\frac{x^{2}}{\lambda}-\frac{2\kappa x}{\lambda^{2}}+\frac{2\kappa^{2}+\lambda}{\lambda^{3}},\quad\textrm{on}\ [0,\infty).
\end{equation}

Moreover, we may verify that for any $x>0$, $v'(x)>0$. Indeed if
$f(x)=x$, then
\[
v'(x)=-\frac{1}{\lambda}e^{(\kappa-\sqrt{\kappa^{2}+2\lambda})x}+\frac{1}{\lambda}=\frac{1}{\lambda}\left(1-e^{(\kappa-\sqrt{\kappa^{2}+2\lambda})x}\right)>0,\quad\textrm{on}\ [0,\infty),
\]
and if $f(x)=x^{2}$, then
\[
v'(x)=\frac{2\kappa}{\lambda^{2}}e^{(\kappa-\sqrt{\kappa^{2}+2\lambda})x}+\frac{2x}{\lambda}-\frac{2\kappa}{\lambda^{2}}.
\]
Since the sign of $v'(x)$ cannot be seen directly, we look at the
second derivative $v''(x)$, that is, 
\[
\begin{aligned}v''(x) & =\frac{2\kappa}{\lambda^{2}}(\kappa-\sqrt{\kappa^{2}+2\lambda})e^{(\kappa-\sqrt{\kappa^{2}+2\lambda})x}+\frac{2}{\lambda}\\
 & =\frac{\sqrt{\kappa^{2}+2\lambda}-\kappa}{\lambda^{2}}\left[(\sqrt{\kappa^{2}+2\lambda}+\kappa)-2\kappa e^{(\kappa-\sqrt{\kappa^{2}+2\lambda})x}\right]\\
 & \geq\frac{\sqrt{\kappa^{2}+2\lambda}-\kappa}{\lambda^{2}}\left[2\kappa\left(1-e^{(\kappa-\sqrt{\kappa^{2}+2\lambda})x}\right)\right]>0.
\end{aligned}
\]
Therefore $v'(x)>0$ for $x>0$. Thus, we know that the optimal control
$u_{t}^{*}$ is the following feedback law by the proof of Theorem
\ref{thm:stochastic-control}: 
\begin{equation}
u_{t}^{*}=-\kappa\mathrm{sgn}(v'(X_{t}))=-\kappa\mathrm{sgn}(X_{t})=-\kappa\in\mathscr{U}.
\end{equation}
Hence the optimal controlled diffusion process with reflection at
zero is then 
\begin{equation}
X_{t}=x+W_{t}-\kappa\int_{0}^{t}\mathrm{sgn}(X_{s})ds+L_{t}=x+W_{t}-\kappa t+L_{t}.
\end{equation}
It is the same process as in (\ref{eq:reflected-bangbang}). For this
case, we have obtained the explicit form of the transition probability
density function $q^{\kappa}(t,x,z)$ of $X_{t}$ in the section \S\ref{sec:Reflecting-bang-bang-diffusion}.
That is,
\begin{equation}
\begin{aligned}q^{\kappa}(t,x,z) & =\frac{1}{\sqrt{2\pi t}}\left[e^{-\frac{(x-z-\kappa t)^{2}}{2t}}+e^{-\frac{(x+z+\kappa t)^{2}}{2t}+2\kappa x}\right]\\
 & \quad\quad\quad\quad\quad\quad+2\kappa e^{-2\kappa z}\int_{x+z-\kappa t}^{+\infty}\frac{1}{\sqrt{2\pi t}}e^{-\frac{u^{2}}{2t}}du,
\end{aligned}
\end{equation}
for any $x\geq0$ and $z\geq0$. 

We would like to point out that, for the problems without reflecting
barriers, similar formulas have been obtained in \cite{KS84,KS91}.

\end{document}